\documentclass[12pt,leqno,twoside]{amsart}
\usepackage[utf8]{inputenc}
\usepackage[T1]{fontenc}
\usepackage[colorlinks=true, pdfstartview=FitV, linkcolor=blue, citecolor=blue, urlcolor=blue]{hyperref}
\usepackage{amstext,amsmath,amscd, bezier,indentfirst,amsthm,amsgen,enumerate, geometry,mathrsfs}
\usepackage[all,knot,arc,import,poly]{xy}
\usepackage{amsfonts,color, soul}  %\st{} for overstriking
\usepackage{amssymb}
\usepackage{latexsym}
\usepackage{epsfig}
\usepackage{graphicx}
\usepackage{srcltx}
\usepackage{enumitem}

\newtheorem{theorem}{Theorem}[section]
\newtheorem{corollary}[theorem]{Corollary}

\newtheorem{proposition}[theorem]{Proposition}

\theoremstyle{definition}
\newtheorem{definition}[theorem]{Definition}
\theoremstyle{remark}
\newtheorem{remark}[theorem]{\sc Remark}
\newtheorem{example}[theorem]{\sc Example}
\renewcommand{\Box}{\square}    %\diamond

\newcommand{\Sing}{{\mathrm{Sing\hspace{2pt}}}}

\newcommand{\rank}{{\mathrm{rank\hspace{1pt}}}}
\newcommand{\Disc}{{\mathrm{Disc\hspace{2pt}}}}

\newcommand{\im}{{\mathrm{Im\hspace{2pt}}}}

\newcommand{\Fib}{{\mathrm{Fib}}}

\newcommand{\e}{\varepsilon}
\newcommand{\m}{\setminus}
\newcommand{\s}{\subset}

\newcommand{\fin}{\hspace*{\fill}$\Box$\vspace*{2mm}}

\newcommand{\cS}{{\mathcal S}}

\newcommand{\cW}{{\mathcal W}}

\newcommand{\bR}{{\mathbb R}}
\newcommand{\bC}{{\mathbb C}}
\newcommand{\bK}{{\mathbb K}}

\begin{document}

\title{On singular maps with local fibration}
%Map germs: the possibility of a fibration

\author{\sc Ying Chen}
\address{School of Mathematics and Statistics, HuaZhong University of Science and Technology WuHan 430074, P. R. China}
\email{ychenmaths@hust.edu.cn}

\author{Mihai Tib\u{a}r}
\address{Univ. Lille, CNRS, UMR 8524 -- Laboratoire Paul Painlev\'e, F-59000 Lille,
France}
\email{mihai-marius.tibar@univ-lille.fr}

\dedicatory{To the memory of Mihnea Col\c toiu}


%\date{July 5, 2019}

\thanks{Ying Chen acknowledges the support from the National Natural Science Foundation of China (NSFC) (Grant no. 11601168). Mihai Tib\u ar acknowledges the support of the Labex CEMPI
(ANR-11-LABX-0007-01). }

\begin{abstract}

We discuss the most general condition under which a singular local tube fibration exists. We give an application to composition of map germs.

\end{abstract}

\maketitle

\section{Stating a meaningful local fibration theorem}

 Let $G :(\mathbb K^m ,0) \to (\mathbb K^p ,0)$ be an analytic map germ, where $m\ge p\ge 1$ and $\bK = \bR$ or $\bC$.
In case $p=1$ and $\bK = \bC$, it is well-known that holomorphic function germs have a local Milnor fibration. Whenever $\bK = \bR$, removing the origin disconnects the real line and, by similar arguments, one obtains a Milnor fibration $G :(\mathbb R^m ,0) \to (\bR ,0)$ which is a trivial fibration over each of the two connected components of the set germ $(\bR \m \{0\}, 0)$.

 The Milnor's local fibration theorem  for holomorphic function germs \cite{Mi} makes sense only if it is independent on the choices of $\e$ and $\delta$ (see Definition \ref {d:tube}). In case of map germs, complex or real, the main problem is to produce the necessary conditions for the existence of the local fibration independently on the chosen neighbourhoods, provided they are  small enough.
 Some part of the recent literature misses the hot core of the local fibration problem by not caring about this independency condition.  
 
 It is of course the setting $p\ge 2$ to which we focus here. 
   %It actually turns out that the existence of a local fibration  \cite{ART}, \cite{ART2}, \cite{JT2}, \cite{JT1}.
 Let us look at two examples in order to get a first impression on  the problems that may occur when varying neighbourhoods,
 and on their complexity.

\begin{example}\label{ex:1}\cite{Sa}.
 $F: (\bC^{3}, 0) \to (\bC^{2}, 0)$, $(x,y,z) \mapsto (x^{2}- y^{2}z,  y)$.   The singular locus of $F$ is included in the central fibre  $F^{-1}(0,0)$.

The image of $F$ is open in 0, namely  we have the equality of set germs $(\im F ,0)  = (\bC^{2},0)$.
Nevertheless this map germ $F$ does not have a locally trivial fibration over the set germ $(\bC^{2}\m \{ 0\},0)$, as shown by Sabbah \cite{Sa}. In particular one cannot apply Ehresmann Theorem due to the local non-properness of the map.
\end{example}

\begin{example}\label{ex:2}
$F: (\bK^{2}, 0) \to (\bK^{2}, 0)$,  $F(x,y) = (x, xy)$. The image $F(B_{\e})$  of the ball $B_{\e}$ centred at 0  depends heavily on its radius $\e>0$. Here, not only that the image $F(B_{\e})$ does not  contain a neighbourhood of the origin, but it turns out that the image of $F$ is not even well-defined as a set germ.
 \end{example}

We discuss here the sharpest conditions under which the analytic map germs  define local fibrations.
We find a general condition under which the composition of map germs has a local fibration. Our main result is Theorem \ref{t:composed},  and we discuss one of its consequences.

%%%%%%%%%%%%%%%%%%%%%%%%%%%%
\section{Tame map germs}\label{s:mapgerms}
\subsection{Map germs having germ image sets}
In Example \ref{ex:2}, the image of the map germs $G$ is not well-defined as a set germ.   Let us define this condition more carefully. We state this in case $\bK = \bR$ but everything holds over $\bC$.
%%%%%%%%%%%%%

\begin{remark}\label{r:germeq}
Let $A, A'\s \bR^{p}$ be two subsets containing the origin. By definition, one has the equality of set germs $(A,0) = (A',0)$ if and only if there exists some open ball $B_{\e}\s \bR^{p}$ centred at 0 and of radius $\e>0$ such that $A\cap B_{\e} = A'\cap B_{\e}$.
\end{remark}
\begin{definition}\label{d:nice}(\emph{Nice map germs}, \cite{ART}, \cite{JT2}.)
Let $G:(\bR^{m},0) \to (\bR^{p}, 0)$, $m\ge p >0$, be a continuous map germ.
 One says that the image  $G(K)$  of a subset $K \subset \bR^{m}$ containing $0$  is a \emph{well-defined set germ}
 at $0 \in \bR^{p}$ if for all small enough open balls $B_{\e_{1}},  B_{\e_{2}}$ centred at 0, with  $\e_{1}, \e_{2}>0$,  the equality of germs
 $(G(B_{\e_{1}} \cap  K), 0) =  (G(B_{\e_{2}} \cap  K), 0)$ holds.

Whenever both images $\im G$ and  $G(\Sing G)$ are well-defined as set germs,  one says that $G$ is a \emph{nice map germ}, and we abbreviate this by writing  NMG.  
\end{definition}

It turns out that the NMG condition implies that the \emph{discriminant}  $\Disc G := G(\Sing G)$ is a well-defined closed subanalytic set germ, see  \cite{JT2}.

Let us point out the following significant cases  of a nice map germ, complex and real, respectively:
\begin{proposition}[\cite{JT2}]\label{p:dim}
Let $F:(\bC^n, 0)\to(\bC^p,0)$,  $n\geq p>0$, be a holomorphic map germ. If the fibre  $F^{-1}(0)$ has  dimension $n-p$
 then $(\im F, 0) = (\bC^{p}, 0)$.  If moreover $\Sing F \cap F^{-1}(0) = \{ 0\}$ then $F$ is a NMG.
 \end{proposition}

\begin{proof}
We fix be some small enough open ball at the origin $B_{\e}\subset \bC^n$  where the holomorphic map $F$  is well-defined.
If $Z\subset B_{\e}$ be a general complex $p$-plane passing through 0, then $0$ is  an isolated point of the slice $Z\cap F^{-1}(0)$.
For any small enough open neighbourhood  $U_{\e}$ of $0$ in $Z$, the induced map $F_{|U_{\e}}:U_{\e}\to \bC^{p}$ is then finite-to-one, and by the Open Mapping Theorem,  this implies
that $F(U_{\e})$ is open, which shows the equality of germs $( \im F,  0) = ( \bC^{p}, 0)$.

 The condition $\Sing F \cap F^{-1}(0) = \{ 0\}$ mens that the set germ $( F^{-1}(0) , 0)$ is  an ICIS = \textit{isolated complete intersection singularity}, and in this case Looijenga proved  in \cite{Lo} that the image $F(\Sing F)$ is a hypersurface germ at $0\in \bC^p$.
\end{proof}

%%%%%%%%

One has the following real counterpart  of the last assertion of Proposition \ref{p:dim}:

\begin{proposition}[\cite{JT2}]\label{l:nice}
Let $G:(\bR^{m},0) \rightarrow (\bR^{p}, 0)$, $m\ge p >0$ be an analytic map germ.
 If $\Sing G \cap G^{-1}(0) = \{ 0\}$,  then  $G$ is a NMG.
 \fin
\end{proposition}

The question whether a given map germ has a well defined image as a set germ seems to be wide open.
  A first  classification has been given in \cite{JT1} for the case case of holomorphic map germs with target $(\bC^{2}, 0)$.

%The condition $\Sing G \cap G^{-1}(0) = \{0\}$ is known to be equivalent to  the finite $C^0$-$\mathcal{K}$-determinacy of %$G$, see \cite{Wa,CB} for details.

%%%%%%%%%%%%%%%%%%%%
%%%%%%%%%%

%%%%%%%%%%%%
%%%%%%%%%%%%
% \section{The existence of fibrations for map germs}
 \subsection{Milnor set and tame map germs}\label{s:nmg}
 %%%%%%%%%%%%%%%%

 We will give here a convenient condition which implies NMG in full generality. 

Let  $G:(\bR^{m},0) \rightarrow (\bR^{p}, 0)$ be a non-constant  analytic map germ, $m \ge p \ge1$.
Let $U \subset \bR^m$ be a manifold, let
	  $\rho := \| \cdot \|$ be the Euclidean distance function, and let $\rho_{|U}$ be its restriction to $U$.
	The set of \textit{$\rho$-nonregular points} of $G_{|U}$, also called \emph{the Milnor set of $G_{|U}$},
	is defined as:	
 $$M(G_{|U}):=\left\lbrace x \in U \mid \rho_{|U} \not\pitchfork_x G_{|U} \right\rbrace .$$

	    It turns out from the definition that $M(G_{|U})$ is real analytic whenever $U$ is supposed analytic.
	By definition $M(G_{|U})$ coincides with the singular set  $\Sing (\rho, G)_{|U}$, which is itself defined as
the set of points $x\in U$ such that  either  $x\in \Sing(G_{|U})$,   or $x\not\in \Sing(G_{|U})$ and $\rank_{x}( \rho_{|U}, G_{|U}) = \rank_{x}(G_{|U})$.   We will actually consider the germ at 0 of $M(G_{|U})$, and we will denote it by $M(G)$.

\begin{definition}[\emph{The Milnor set in the stratified setting}]\label{d:Mstr} \
Let  $G:(\bR^{m},0) \rightarrow (\bR^{p}, 0)$, with $m\geq p >1$, be a non-constant  analytic map germ. We say that  a finite semi-analytic Whitney (a)-regular stratification $\cW$ of $\bR^{m}$ is a stratification of $G$ if $\Sing G$ is a union of strata, and  such that the restriction $G_{|W}$ has constant rank  for any $W\in \cW$.

Let $W \in \cW$  be the germ at 0 of the stratum $W$, and let
$M(G_{|W})$ be the Milnor set of $G_{|W}$, as defined above.
One then calls
\[ M(G):=\bigsqcup_{W\in \cW} M(G_{|W_\alpha}) \]
	the set of \textit{stratwise $\rho$-nonregular points} of $G$ with respect to the stratification $\cW$.
\end{definition}

\begin{remark}\label{r:milnorset}
The Milnor set  $M(G)$  is  closed because $\cW$ is a Whitney (a)-stratification. Also notice that if $\rank G_{|W}= \dim W$, then $W\subset M(G)$. 

Because of the next definition, we are especially interested in the intersection $M(G)\cap G^{-1}(0)$. The following inclusion holds:
\begin{equation}\label{eq:milnorradius}
 M(G)\cap G^{-1}(0) \subset \Sing G \cap G^{-1}(0).
\end{equation}
Indeed, by Milnor's classical result on the local conical structure of semi-analytic sets \cite{Mi}, \cite{BV}, there exists $\e_{0}>0$ such that the manifold $G^{-1}(0) \m \Sing G$ is transversal to the  sphere $S^{m-1}_{\e}$ centred at 0,  for any $0<\e <\e_{0}$. For any fixed point $a\in G^{-1}(0) \m \Sing G$,  a whole open ball $B$ centred at $a$ does not intersect  $\Sing G$, and it then follows that the nearby fibres of $G$ inside $B$
  are also transversal to the levels of the distance function $\rho$, provided that $B$ is small enough.

This implies that $M(G)\cap (G^{-1}(0) \m \Sing G) =\emptyset$,
which proves our claim.
\end{remark}

\begin{definition}[\emph{Tame map germs, \cite{JT2}}]\label{d:tame}
Let $G:(\bR^{m},0) \rightarrow (\bR^{p}, 0)$, with $m > p \ge 2$, be a non-constant  analytic map germ. We say that $G$ is \emph{tame} with respect to the stratification $ \cW$ if the following  inclusion of set germs holds:
\begin{equation}\label{eq:main2}
\overline{M(G) \m  G^{-1}(0)} \cap G^{-1}(0) \subset     \{ 0\}.
\end{equation}
 \end{definition}

It follows from the definition that if $G$ is tame then the closure of the strata of $ \cW$ of
 dimensions $\le p$ intersect  $G^{-1}(0)$ only at $\{ 0\}$.

\begin{remark}
The identity map is not tame.
 As one may easily check, Examples \ref{ex:1} and \ref{ex:2} are not tame.  
\end{remark}

As we work in a highly singular situation,  we have to insure the existence of the images as set germs because this is a preliminary condition in the  problem of the existence of a singular local fibration. The ``tame'' condition just does this job:

\begin{theorem}\label{main-new}\cite{JT2}
Let  $G:(\bR^{m},0) \rightarrow (\bR^{p}, 0)$, with $m \geq p \geq 2$, 
be a non-constant  analytic map germ.
If $G$ is tame  then:
\begin{enumerate}
 \rm \item \it $G$ is a NMG at the origin.
 \rm \item \it the image  $G(W_\alpha)$ is a well-defined set germ at the origin, for any stratum $W_\alpha \in  \cW$.
 %\rm \item \it  $\Disc G = G(\Sing G)$.
\end{enumerate}
\fin
\end{theorem}

\begin{remark}\label{r:condtame}
Let us point out two very particular cases where $\im G$ is well-defined as a set germ:

\noindent
(i). $\Sing G \subset  G^{-1}(0)$ and  $G$ satisfies condition \eqref{eq:main2}.
It was shown  by Massey \cite{Ma} that  under these hypotheses one has $(\im G,0) = (\bR^p,0)$.

\noindent
(ii).  $\Sing G \cap G^{-1}(0) \not= G^{-1}(0)$.
This is the hypothesis of  Proposition \ref{l:nice}, and   condition  \eqref{eq:main2} is not required. In this case we also have  $(\im G,0) = (\bR^p,0)$.
\end{remark}

\subsection{Singular stratified fibration theorem}\label{s:sing}
%%%%%%%%%%%%%%%%%%%%%

%\section{Existence of Milnor-Hamm tube fibration}
%%%%%%%%%%

After  \cite{JT2}, we show that \emph{tame} is the most suitable condition under which  one can prove the existence of a local singular fibration.
We focus on the general case $\dim \Disc G >0$. We first need two definitions.

\begin{definition}[Regular stratification of the map germ $G$]\label{d:regularstr}
Let  $G:(\bR^{m},0) \rightarrow (\bR^{p}, 0)$ be a non-constant  analytic map germ, $m> p \ge 2$, and let
$ \cW$ be a Whitney (b)-regular stratification  of $G$ at 0, as defined above.
We assume that $G$ is  tame.
Then  Theorem \ref{main-new}   tells that the
images of all strata of $ \cW$ are well-defined as set germs at 0.
By the classical stratification theory,
there exists a germ of a finite subanalytic stratification $\cS$   of  the target such that  
$\Disc G$ (which is a closed subanalytic set, as remarked after Definition \ref{d:nice})
 is a union of strata,  and such that $G$ is  a stratified submersion relative to the couple 
 of stratifications $(\cW,  \cS)$,
 meaning that the image by $G$ of a stratum  $W_\alpha \in  \cW$ is  a single stratum $S_{\beta} \in \cS$,
and that the restriction $G_{|}:W_\alpha \to S_{\beta}$ is a submersion.
	
We call the couple  $(\cW,  \cS)$ a \emph{regular stratification of the map germ $G$.}
\end{definition}		

\begin{definition}[Singular Milnor tube fibration, \cite{JT2}]\label{d:tube}
Let $G:(\bR^{m},0) \rightarrow (\bR^{p}, 0)$, $m\ge p>1$,  be a non-constant analytic map germ. Assume that there exists some regular stratification $( \cW,  \cS)$ of $G$.

We say that $G$ has  a \emph{(singular) Milnor tube fibration} relative to $(\cW,  \cS)$  if for any small enough $\e > 0$ there exists  $0<\eta \ll \e$ such that the restriction:
\begin{equation}\label{eq:tube1}
G_| :  B^{m}_{\e} \cap G^{-1}( B^{p}_\eta \m \{ 0\} ) \to  B^{p}_\eta \m \{ 0\}
\end{equation}
is a stratified locally trivial fibration which is independent, up to stratified 
homeomorphisms, of the choices of $\e$ and $\eta$.

By \emph{stratified locally trivial fibration} we mean that for any stratum $S_{\beta}$ of $ \cS$, the restriction $G_{| G^{-1}(S_{\beta})}$ is a locally trivial \emph{stratwise fibration}.
\end{definition}

What means more precisely ``independent, up to stratified homeomorphisms, of the choices of $\e$ and $\eta$''?
 It means that when replacing the Milnor data $(\e, \eta)$ by another Milnor data $(\e', \eta')$, for  some $\e'<\e$, and suitably smaller $\eta'<\eta$, then the two  fibrations  \eqref{eq:tube1} 
have the same stratified image in the smaller ball $B^{p}_{\eta'}$, and these fibrations are stratified diffeomorphic 
over $B^{p}_{\eta'} \m \{ 0\}$.

\medskip

The non-empty fibres of \eqref{eq:tube1} are those over some connected stratum $S_{\beta} \subset \im G$ of $ \cS$, and
 such fibre  $G^{-1} (s)$, for some $s\in S_{\beta}$, is the union $\sqcup_{W_{\alpha}}G^{-1} (s)$  over all strata $W_{\alpha}\subset G^{-1} (S_{\beta})$.

One has the following fundamental result:
\begin{theorem} \label{t:tube}\cite{JT2}
Let $G:(\bR^m, 0) \to (\bR^p,0)$, $m > p\ge 2$,  be a non-constant analytic map germ. 	
If $G$ is tame, then $G$ has a singular Milnor tube fibration \eqref{eq:tube1}.
\fin
\end{theorem}

The  proof consists in showing that  the restriction map
\begin{equation}\label{eq:restrG}
G_{|}:  W_{\alpha}\cap \overline{B^m_\e} \cap G^{-1}(S_{\beta} \cap B^{p}_\eta \m \{0\}) \to  S_{\beta} \cap  B^{p}_{\eta} \m \{0\}
\end{equation}
is well-defined, it is  a submersion from a manifold with boundary, hence it is therefore a locally trivial stratified fibration by Thom-Mather Isotopy Theorem. The tame condition \eqref{eq:main2} also implies that this fibration is independent of $\e$ and $\eta$ up to stratified homeomorphisms.

\begin{remark}
 The relation between \emph{tame} and the Thom regularity illustrates once more why \emph{tame} is a universal condition, more precisely we have the implication:  if $G$ has a Whitney stratification $\cW$ (cf Definition \ref{d:Mstr}) such that it is Thom regular at all the strata  included in $G^{-1}(0)$ then $G$ is tame.
  We refer to \cite{JT2} for details, and for several examples.
\end{remark}

%%%%%%%%%%%%%%%
\section{Applications of the tube fibration}\label{s:appl}

 We consider here the particular setting of analytic map germs $G:  (\bR^m, 0) \to (\bR^p, 0)$ with \emph{isolated singular value}, which means $\Sing G \subset G^{-1}(0)$. This is more general than what was usually considered in the literature after Milnor's \cite{Mi} and before 2008, namely map germs with an isolated singularity. With \cite{AT0}, \cite{AT1}, \cite{Ma} and \cite{ACT} this more general setting started to be studied in more detail. 
 %False approach by fixing local data $(\e, \delta)$? 
 
\begin{remark}\label{r:pointout}
  A map germ $G:  (\bR^m, 0) \to (\bR^p, 0)$, $m\ge p\geq 2$, with \emph{isolated singular value} has the following properties:
\end{remark}

 \begin{itemize}
\item[(a).]   The stratification of the target has only two strata: $\bR^p \m \{0\}$, and the origin.
 A stratification of the source means some Whitney (a)-regular stratification of the singular locus $\Sing G \subset G^{-1}(0)$ without any other conditions, and $\bR^m \m \Sing G$ is one stratum, if connected, or  it is the union of the connected components of $\bR^m \m \Sing G$ as strata.

\item[(b).]  The \emph{tame} condition for $G$ amounts to the \emph{$\rho$-regularity} in our setting of isolated singular value, which reads (compare to Definitions \ref{d:Mstr} and \ref{d:tame}):
\begin{equation}\label{eq:rhoreg}
  \overline{M(G_{|\bR^m \m G^{-1}(0)})} \cap G^{-1}(0) \subset     \{ 0\}.
\end{equation}
Whenever \eqref{eq:rhoreg} holds, we will say for short: \emph{$G$ is $\rho$-regular}.

\item[(c).]  If the map germ $G$ is tame and with isolated singular value, then the image of $G$ is locally open,
in other words we have the equality of set germs $(\im G, 0)=(\bR^p, 0)$. This was remarked by Massey \cite[Cor. 4.7]{Ma},  see also \cite{AT1}, \cite{JT2}.  
Moreover,  since $G$ has tube fibration by Theorem \ref{t:tube}, and since $\bR^p \setminus \{0\}$ is connected for $p\ge 2$, the Milnor fibre $\Fib (G)$ of $G$ is independent of the point, in the sense that one has a diffeomorphism $G^{-1}(a) \simeq G^{-1}(b)$ for any $a,b\not= 0$ in a small disk at 0 in $\bR^p$.
 \end{itemize}
 
   \
Theorem \ref{t:tube} enables us to treat some significant classes of singular map germs, as follows.  

%%%%
\begin{theorem}\label{t:composed}
Let $F:(\mathbb{R}^m,0)\rightarrow(\mathbb{R}^p,0)$ and $G:(\mathbb{R}^p,0)\rightarrow(\mathbb{R}^k,0)$, $m\geq p\geq k\geq 2$, be analytic map germs, such that $F$ is tame and has isolated singular value, and that $G$ has an isolated singular point at the origin. % by Milnor has a Milnor tube fibration (cf Definition \ref{d:tube}). 

Then  $H = G\circ F$ is tame, locally open, and has a local tube fibration.
\end{theorem}

\begin{proof}
 Let us first remark that $G$ has a tube fibration since it has an isolated singular point, as Milnor has shown in \cite[page 94]{Mi}, see also \cite{ACT}. 
 By comparing the corresponding Jacobian matrices, we deduce that the hypotheses $\Sing F \subset F^{-1}(0)$ and $\Sing G\subset \{0\}$ imply $\Sing H \subset F^{-1}(0)\subset H^{-1}(0)$. This shows in particular that $H$ has isolated singular value. It moreover shows:
\begin{equation}\label{eq:incl0}
 \overline{M(H)\setminus H^{-1}(0)}\cap H^{-1}(0) \subset M(H)\cap H^{-1}(0)  \subset \Sing H \cap H^{-1}(0)   \subset \Sing H \cap F^{-1}(0),
\end{equation}
where the first inclusion is obvious, the second is shown in \eqref{eq:milnorradius} of Remark \ref{r:milnorset}, and the last inclusion was pointed out just above.

 We have the inclusions $M(H)\subset M(F)$ and $F^{-1}(0)\subset H^{-1}(0)$, and therefore we obtain the inclusion
 $M(H)\setminus H^{-1}(0) \subset M(F)\setminus F^{-1}(0)$.
 
 By taking closures and intersecting with $H^{-1}(0)$, we then get the inclusions:
 \[ \overline{M(H)\setminus H^{-1}(0)}\cap H^{-1}(0) \subset  \overline{M(H)\setminus H^{-1}(0)}\cap F^{-1}(0)\subset  \overline{M(F)\setminus F^{-1}(0)}\cap F^{-1}(0)
 \]
 where the first one uses \eqref{eq:incl0}.  This shows
  that the $\rho$-regularity of $F$ implies the $\rho$-regularity of $H$.  By the Tube Fibration Theorem \ref{t:tube},  the map  $H$ has a tube fibration. It is also locally open due to Remark \ref{r:condtame}(i).
\end{proof}

%%%

One has the following particular case of the above results:

\begin{corollary}\label{c:Milnorquest}
Let $F=(g_{1},\ldots , g_{p}): (\bR^m, 0) \to (\bR^p, 0)$, $m\ge p\geq 3$,  be an analytic map with an isolated singular value, and $\rho$-regular.

Let $H := (g_{1},\ldots,g_{p-1}): (\bR^m, 0) \to (\bR^{p-1}, 0)$.  Then $H$ has a tube fibration and its Milnor fibre of $H$ is homeomorphic to the Milnor fibre of $F$ times an open interval.
\fin
 \end{corollary}
 
  In the particular setting ``$F$ has an isolated singular point'',  Corollary  \ref{c:Milnorquest} has been  Milnor's question at \cite[p.100]{Mi}. In the above more  general setting ``$F$ has an isolated critical value'', a proof of Corollary  \ref{c:Milnorquest} has been  given in \cite[Theorem 6.3]{AD}.

\begin{proof}[Proof of Corollary \ref{c:Milnorquest}]
The map germ $G$ from Theorem \ref{t:composed} is here the projection $(\bR^p, 0) \to (\bR^{p-1}, 0)$ to the first $p-1$ coordinates.  As a direct consequence of Theorem \ref{t:composed}, it follows that   $H= F\circ G$ has a tube fibration. 

The image by $F$ of the Milnor fibre of $H$ is a line segment $I$, which is contractible.  The restriction of $F$ on the Milnor fibre of $H$ is a locally trivial fibration of  base $I$ and fibre the Milnor fibre of $F$, thus it is a trivial fibration.
\end{proof}

\begin{example}
Let $F:(\mathbb{R}^3,0)\rightarrow(\mathbb{R}^2,0)$, $F(x,y,z)=(y^{4}-z^{2}x^{2}-x^{4}, xy)$. It was shown in \cite[Example 5.1]{ACT} that $F$ is tame with isolated singular value.  A direct computation shows that the Milnor fibre of $F$ is a disjoint union of two closed intervals, for instance by considering  $F^{-1}(\delta,0)$ for $\delta>0$ sufficiently small.

Let $G:(\mathbb{R}^2,0)\rightarrow(\mathbb{R}^2,0)$, $G(u,v)=(u^{2}-v^{2}, 2uv)$. We see that $G$ has an isolated singular point at the origin, thus it has a Milnor  tube fibration, and one easily computes that the Milnor fibre is two points. By Theorem
\ref{t:composed}, the map germ $H = G\circ F$ has a local tube fibration. Its Milnor fibre is the disjoint union of 4 line segments.
\end{example}

%%%%%%%%%%%%%%%%%%%%%%%%%%%

%%%%%%%%%%%%%%%%%%%%%
\end{document}